\documentclass[12pt]{amsart}
\usepackage[utf8]{inputenc}
\usepackage{amsmath,amsthm,amsfonts,amscd,amssymb,amsopn,enumerate,todonotes}
\usepackage{eucal, mathrsfs,color}
\usepackage{pxfonts,mathabx}
\usepackage[all]{xy}
\usepackage{hyperref}
\usepackage[capitalise]{cleveref}
\usepackage{fullpage}
\newcommand{\mbb}[1]{\mathbb #1}

\newcommand{\ms}[1]{\mathscr #1}

\newcommand{\ov}[1]{\overline{#1}}
\newcommand{\til}[1]{\widetilde{#1}}

\theoremstyle{plain}

\newtheorem{thm}{Theorem}[section]
\newtheorem{defn}[thm]{Definition}
\newtheorem{notn}[thm]{Notation}

\newtheorem{lem}[thm]{Lemma}

\newtheorem{prop}[thm]{Proposition}
\newtheorem*{thm*}{Theorem}
\newtheorem*{rem*}{Remark}
\newtheorem*{lem*}{Lemma}

\newtheorem*{cor*}{Corollary}
\newtheorem*{prop*}{Proposition}
\newtheorem{rem}[thm]{Remark}
\newtheorem{question}[thm]{Question}
\newtheorem{ex}[thm]{Example}

\theoremstyle{remark}

\newcommand{\op}[1]{\operatorname{#1}}

\newcommand{\Br}{\op{Br}}
\newcommand{\Spec}{\op{Spec}}

\newcommand{\Pic}{\op{Pic}}

\newcommand{\ind}{\op{ind}}
\newcommand{\im}{\op{im}}
\newcommand{\per}{\op{per}}

\newcommand{\id}{\op{id}}

\DeclareFontFamily{U}{wncy}{}
    \DeclareFontShape{U}{wncy}{m}{n}{<->wncyr10}{}
    \DeclareSymbolFont{mcy}{U}{wncy}{m}{n}
    \DeclareMathSymbol{\Sha}{\mathord}{mcy}{"58}
    \DeclareMathSymbol{\B}{\mathord}{mcy}{"42}

\newcommand{\precspl}[1]{\prec_{#1}}

\newcommand{\clK}{\underline{\mathscr K}}

\title{Transcendental splitting fields of division algebras}

\author{Daniel Krashen and Max Lieblich}

\begin{document}

\begin{abstract}
We examine when division algebras can share common splitting fields of certain types. In particular, we show that one can find fields for which one has
infinitely many Brauer classes of the same index and period at least 3, all nonisomorphic and having the same set of finite splitting fields as well as the same splitting fields of transcendence degree $1$ and genus at most $1$. On the other hand, we show that one fixes any division algebra over a field, then any division algebras sharing the same splitting fields of transcendence degree at most 3 must generate the same cyclic subgroup of the Brauer group. In particular, there are only a finite number of such division algebras. We also show that a similar finiteness statement holds for splitting fields of transcendence degree at most $2$.
\end{abstract}

\maketitle

\section{Introduction}

In this paper, we pursue the question of how much information about a Brauer class can be obtained from knowledge of its splitting fields of various types.
In recent years, we have seen a number of examples in the literature both
of division algebras sharing all of their maximal subfields, and even
sharing all of their finite dimensional splitting fields without being
isomorphic or without being powers of one another in the Brauer group (see \cite{GarSalt,RR,CRR1, CRR2, CRR3, CRR4, CRR5,KrMcK,KMRRS}).

To describe these kinds of results, it is useful to consider partial orders on the collection of Brauer classes. For a field $k$, and a collection of field extensions $\ms K$ of $k$ we write $\alpha_1 \precspl{\ms K} \alpha_2$ for Brauer classes $\alpha_1, \alpha_2 \in \Br(k)$ if for every $K \in \ms K$ splitting $\alpha_1$, we have that $K$ also splits $\alpha_2$.
We can then consider the ``upper genus of $\alpha_1$ with respect to $\ms K$'', written $[\alpha_1, \infty)_{\ms K}$ consisting of all the classes $\alpha_2$ such that $\alpha_1 \precspl{\ms K} \alpha_2$. Note that this is a subgroup of the Brauer group of $k$.

How does one expect this to behave in general? In the particular case of the number fields and finite field extensions, this is easy to describe -- a field extension $E/F$ of a number field $F$ splits a Brauer class $\alpha$ whenever its composita $E F_v$ with respect to completions $F_v$ of $F$ have local degree a multiple of the ramification degree (order of the Hasse invariant) of $\alpha$ at $v$. In particular, one then finds that $[\alpha, \infty)$ consists of exactly those Brauer classes whose Hasse invariants divide those of $\alpha$. In particular, this is a finite group.

It is natural to ask in which general situations these upper genera are finite. We show that this holds if one considers the collection of finitely generated field extensions of transcendence degree at least 2 (Theorem~\ref{2 genus}) over a field of characteristic $0$, and in the case of transcendence degree at least 3, it consists of precisely the cyclic subgroup generated by the Brauer class (Theorem~\ref{3 genus}).

From the other direction, it is natural to ask how large upper genera can be for field extensions of transcendence degree at most $1$. For finite extensions, examples of infinite upper genera are given in \cite[Theorem~G]{KMRRS}. In general, this question is still open (see \cref{question}), however, as these are represented by function fields of curves, one can subdivide this question according to the genus of the curves in question. For curves of genus $0$, upper genera consist of singletons, as noted in \cref{genus 0}. We show in \cref{forcing 1} that one can construct examples of infinite collections of Brauer classes over a field which share all the same finite splitting fields and splitting fields which are function fields of curves of genus at most $1$.

\bigskip

The authors are very grateful to the anonymous referee for many helpful comments and corrections, in particular in helping us find an error in the original version of this paper, which contained an incorrect proof of an affirmative answer to \cref{question}. 

\section{Preliminaries on obstructions and splittings}

Splittings of Brauer classes by functions fields are governed by specializations of geometric Brauer classes. Therefore, in order to approach our results, it will be useful to first review this connection.

Suppose that $X$ is a smooth, projective, geometrically connected variety over a field $k$.
We may describe the relative Brauer group $\Br(X/k)$ in terms of obstruction classes for line bundles (see also \cite{CiKra,Ma:SBP}). Note that this relative Brauer group can also be identified with $\Br(k(X)/k)$ since $\Br(X) \to \Br(k(X))$ is injective by \cite[Cor.~1.8]{Gro:gdb2}.
Throughout, for any scheme $Y$, we will write $\Br(Y)$ to denote the cohomological Brauer group $H^2(Y, \mathbb G_m)$.

We will have use to extend our discussion to the relative situation, so let us begin by assuming we have a projective morphism $\pi: \ms X \to S$ over some base scheme $S$ such that $\pi$ is cohomologically flat in dimension $0$ (see \cite[page~206]{NeronModels}), that is, so that $\pi_* \ms O_{\ms X} \simeq \ms O_S$ holds universally after base change. For any morphism $T \to S$, we may
consider the Leray-Serre spectral sequence for \'etale cohomology (natural in $T$):
\[H^p(T, R^q\pi_* \mbb G_m) \implies H^{p+q}(\ms X_T, \mbb G_m).\]
This gives an exact sequence of low degree terms
\begin{equation} \label{obstruction sequence}
\Pic_{\ms X_T/T}(T) = \Pic_{\ms X/S}(T) \to \Br(T) \to \Br(\ms X_T)
\end{equation}
where $\Pic_{\ms X/S}$ denotes the sheafification of the Picard functor \cite[page~201]{NeronModels} (note that one can use the \'etale topology instead of the flat topology since $\pi$ is proper as discussed in \cite[page~203]{NeronModels}). We may therefore describe the relative Brauer group $\Br(\ms X_T/T) = \ker\left(\Br(T) \to \Br(\ms X_T)\right)$ as the image of the map from the sections of the Picard sheaf.

In particular, if we consider a smooth geometrically integral projective variety $X$ over a field $k$, the Picard sheaf $\Pic_{X/k}$ is represented by a group scheme locally of finite type over $k$ 
which we will denote $\Pic_X$ (\cite[Chapter~9, Corollary~9.4.18.3]{FGAexplained}), and if we let $T = \Pic_X$, we obtain a natural Brauer class $\beta \in \Br(\Pic_{X})$ such that the map $\Pic_{X}(k) \to \Br(k)$ in sequence~\ref{obstruction sequence} is given by specialization. This can be seen via the map of spectral sequences induced by pulling back with respect to a map $x: \Spec k \to \Pic_X$ (i.e. $x \in \Pic_X(k)$):
\[\xymatrix{
H^p(\Pic_{X}, R^q\pi_* \mbb G_m) \ar[d] \ar@{=>}[r] & H^{p+q}(X_{\Pic_{X}}, \mbb G_m) \ar[d] \\
H^p(k, R^q\pi_* \mbb G_m) \ar@{=>}[r] & H^{p+q}(X, \mbb G_m).
}\]
giving a commutative diagram with exact rows:
\begin{equation} \label{natural obstruction}
\xymatrix{
\Pic_X(\Pic_X) \ar[r] \ar[d] & \Br(\Pic_X) \ar[r] \ar[d] & \Br(X_{\Pic_X}) \ar[d] \\
\Pic_X(k) \ar[r] & \Br(k) \ar[r] & \Br(X).
}
\end{equation}
Since the pullback $x^*: \Pic_X(\Pic_X) \to \Pic_X(k)$ takes the identity morphism $\id_{\Pic_X}$ to $x$, it follows that the image of $\id_{\Pic_X}$ in $\Br(\Pic_X)$ describes a Brauer class on $\Pic_X$ whose specializations at rational points (i.e. pullbacks via $x^*: \Br(\Pic_X) \to \Br(k)$) exactly describe the Brauer classes split when base changed from $k$ to $X$.

\begin{defn}
We call this class $\beta \in \Br(\Pic_X)$ the canonical obstruction class for $X/k$.
\end{defn}
\begin{rem} \label{pic hom}
It follows that the map $\Pic_X(k) \to \Br(k)$ from the bottom of diagram~\ref{natural obstruction} has the interpretation both as the homomorphism coming from the Artin-Leray spectral sequence as well as arising from taking at point $x \in \Pic_X(k)$ to $\beta|_x$, the specialization of the Brauer class $\beta$. Consequently, this class $\beta$ has the property that specialization to rational points gives a homomorphism $\Pic_X(k) \to \Br(k)$.  
\end{rem}

Summarizing this discussion, we have:
\begin{prop} \label{canonical obstruction}
Let $X$ be a smooth projective variety, and let $\Pic_X$ be the Picard scheme for $X$. Then there is a Brauer class $\beta_X \in \Br(\Pic_X)$, called the canonical obstruction class, such that specialization
\begin{align*}
\Pic_X(k) &\to \Br(k) \\
x &\mapsto \beta_X|_x
\end{align*}
induces a homomorphism from $\Pic_X(k)$ to $\Br(k)$ whose image is exactly the relative Brauer group $\Br(k(X)/k) = \Br(X/k)$.
\end{prop}

\section{Specializations of Brauer classes}

As we have remarked in Proposition~\ref{canonical obstruction}, splittings by function fields can be described by specializations of Brauer classes on schemes (i.e. the Picard scheme). For this reason, to understand how to construct situations where certain transcendental field extensions split a Brauer class $\alpha$, we begin by investigating how to construct situations to make a class $\alpha$ a specialization of a particular geometric Brauer class.

\begin{defn} \label{splitting po}
Let $k$ be a field. A \textbf{pointed Brauer class over $k$} is a triple $(\beta, X, x)$ consisting of a geometrically integral variety $X/k$ with a $k$-rational point $x$, and a Brauer class $\beta \in \Br(X)$ such that $\beta|_x = 0$.
\end{defn}
\begin{rem} \label{nonconstant nontrivial}
We note that if $b = (\beta, X, x)$ is a pointed Brauer class, then $\beta \in \im(\Br(k) \to \Br(X))$ if and only if $\beta = 0$, since evaluation at $x$ gives a section, and by hypothesis $\beta|_x = 0$.
\end{rem}
\begin{rem}
If $\Pic_X^0$ denotes the identity component of the Picard variety of $X$, and if this is a smooth projective variety (for example, in the case that $X$ is a curve) then since the identity element of $\Pic_X^0(k)$ is sent to the identity in $\Br(k)$ (by Remark~\ref{pic hom}), it follows that $(\beta_X, \Pic_X^0, 0)$ is a pointed Brauer class.
\end{rem}

\begin{defn}
We let $\B_k$ denote the set of all pointed Brauer classes over $k$. For $b = (\beta, X, x) \in \B_k$, we write $\per(b)$ for $\per(\beta)$ (the period of $\beta$), and if $L/k$ is a field, we write $b_L$ for $(\beta_L, X_L, x_L)$.
\end{defn}

\begin{defn}
Let $k$ be a field and $(\beta, X, x) \in \B_k$ a geometric Brauer class over $k$. We say that $\alpha \in \Br(k)$ is a specialization of $(\beta, X, x)$ and write $\alpha \in (\beta, X, x)$ if $\alpha = \beta|_y$ for some $y \in X(k)$.
\end{defn}

We note that this property is clearly preserved under arbitrary field extensions -- if $\alpha \in b \in \B_k$ then $\alpha_L \in b_L \in \B_L$ for any $L/k$.

We will need to consider how a given Brauer class over $k$ may be obtained by specialization of geometric Brauer classes. As nontrivial geometric Brauer classes are necessarily nonconstant by Remark~\ref{nonconstant nontrivial}, it is not clear whether or not clear, a priori, whether or not a general Brauer class over $k$ can be obtained in this fashion.

\begin{prop} \label{eventually special prop}
Let $k$ be a field, and suppose that $\alpha, \beta \in \Br(k)$, $\per(\beta) | \per(\alpha)$ and $b \in \B_k$ with $\alpha \in b$.
Then we may find an inert field extension $E/k$ such that $\beta_E \in b_E$. 
\end{prop}
This result will follows immediately from the next lemma, observing that $\alpha \in b$ implies $\per \alpha | \per b$ and hence $\per \beta | \per b$ as well.

\begin{lem} \label{eventually special}
Let $k$ be a field, and suppose that $\beta \in \Br(k)$ Brauer class.  Suppose that $b \in \B_k$ with $\per(\beta) \mid \per(b)$. Then we may find a
geometrically integral variety $Y(\beta, b)$ with function field $k(\beta, b)$ such that $\Br(k) \to \Br(k(\beta, b))$ is injective and that $\beta_{k(\beta, b)} \in b_{k(\beta, b)}$.
\end{lem}
\begin{proof}
Write $b = (\delta, X, x)$. Without loss of generality, we may assume that $\beta$ and hence $\delta$ are nonzero. Since $X$ has a smooth $k$-rational point, it follows that $\Br(k)
\to \Br(X) \to \Br(k(X))$ is injective. Consider the Brauer class $\gamma = \delta - \beta_X \in \Br(X)$. Let $Y_{\beta, b} \to X$ be the Severi-Brauer scheme of $\gamma$ and let $k(\beta, b)$ be its function field. Let $x' \in X(k(X))$ be the generic
point of $X$, and $x \in X(k(\beta, b))$ the corresponding point defined by pullback. By \cref{nonconstant nontrivial}, as $\delta$ is nonzero, it and therefore also $\delta_{x'}$ are not in the image of $\Br(k)$. Since $\per(\beta_K) = \per(\beta) | \per(\delta_{x'})$, no nontrivial multiple of $\gamma = \delta_{x'} - \beta_{k(X)}$ is in the image of $\Br(k) \to \Br(k(X))$. As the kernel of the map $\Br(k(X)) \to \Br(k(\beta, b))$ is generated by the class $\gamma$ by \cite{Am}, it follows that the kernel of the composition $\Br(k) \to \Br(k(X)) \to \Br(k(\beta, b))$ is trivial. Finally, since $(\gamma)_{k(\beta, b)} = 0$ by construction, it follows that $\beta_{k(\beta, b)} = \delta_x$ as claimed.
\end{proof}

\section{Splitting Brauer classes with collections of field extensions}

In this section we collect some useful language and observations for working with collections of field extensions and partial orders of Brauer classes by their splitting properties. We write $\per(\alpha)$ to denote the period of a Brauer class $\alpha$.

\begin{defn}
We say that a field extension $L/k$ is \textbf{inert} if $L/k$ is regular (that is, if $L \otimes_k \ov{k}$ is again a field) and if $\Br(L/k) = 0$.  
\end{defn}

\begin{notn}
Let $F/k$ be regular and $E/k$ any field extension. Then $F \otimes_k L$ is a domain, and we write $F\cdot_k L$ for the fraction field of $F \otimes_k L$ (or we will just write $FL$ instead of $F \cdot_k L$ when $k$ is clear from the context).
\end{notn}

\begin{notn}
Let $\ms K$ be a collection of field extensions of $k$, and suppose $F/k$ is regular. We write $\ms K \cdot_k L$ (or $\ms K L$) to denote the collection of fields of the form $K \cdot_k L$ for $K \in \ms K$.
\end{notn}

\begin{defn}
Let $\ms A_1, \ms A_2 \subset \Br(k)$ be nonempty collections of Brauer classes over $k$. We say that $\per(\ms A_1) | \per(\ms A_2)$ if $\per \alpha_1 | \per \alpha_2$ for all $\alpha_1 \in \ms A_1$ and $\alpha_2 \in \ms A_2$. We say that $\per \ms A_1 = \per \ms A_2$ if $\per \alpha_1 = \per \alpha_2$ for every $\alpha_i \in \ms A_i$ (and hence all periods of all Brauer classes in both sets coincide).
\end{defn}

\begin{defn}
For a collection of field extensions $\ms K$ and Brauer classes $\alpha_1, \alpha_2 \in \Br(k)$, we say $\alpha_1 \prec_{\ms K} \alpha_2$ if for every $L \in \ms K$, if $(\alpha_1)_L = 0$ then $(\alpha_2)_L = 0$. For $\ms A_1, \ms A_2 \subset \Br(k)$ we say $\ms A_1 \prec_{\ms K} \ms A_2$ if $\alpha_1 \prec_{\ms K} \alpha_2$ whenever $\alpha_1 \in \ms A_1$, $\alpha_2 \in \ms A_2$. Finally, we write $\ms A_1 \sim_{\ms K} \ms A_2$ if $\ms A_1 \prec_{\ms K} \ms A_2$ and $\ms A_2 \prec_{\ms K} \ms A_1$.
\end{defn}

In other words, $\alpha_1 \prec_{\ms K} \alpha_2$ is every splitting field in $\ms K$ of $\alpha_1$ also splits $\alpha_2$. We will be interested in, for example, $\ms K$ being the finite extensions of $k$, or the finitely generated field extensions of a given transcendence degree, or the function fields of curves of genus $1$. In each of these cases, the class $\ms K$ naturally can be thought of as ``varying'' with $k$ in a natural way. This leads us to the following definition:

\begin{defn}
Let $k$ be a field. A class $\clK$ of field extensions of $k$ is a rule which associates to every field extension $F/k$ a set $\clK(F)$ whose elements are field extensions of $F$, such that if $E \in \clK(F)$ and $L/F$ is any regular field extension, then every $E \cdot_F L$ is (isomorphic to) a field in $\clK(L)$.
\end{defn}

The particular classes of fields we will consider, as have been described already above, satisfy an additional property. Essentially this property says that they are determined by finitely generated subfields.

\begin{defn}
Let $\clK$ be a class of field extensions of $k$. We say that $\clK$ is \textbf{regularly compact} if for any sequence of field extensions $F_0 \subset F_1 \subset F_2 \subset \cdots$ of $k$, with $F_{n + 1} / F_n$ a regular extension for each $n$, and for every $K \in \clK(\bigcup F_n)$, there exists $K' \in \clK(F_m)$ for some $m$ such that $K = K' \cdot_{F_m} (\bigcup F_n)$.
\end{defn}

\begin{ex} \label{genus 1 example}
Let $\clK(F)$ be the collection of field extensions of the form $F(X)$, where $X/F$ is a smooth geometrically integral curve of genus $1$. Then $\clK(F)$ is regularly compact.
\end{ex}
 To see this, suppose we have $K \in \clK(\bigcup F_n)$. Then $K$ is the function field of some curve $X$ which may be written as the zeros to some system of polynomials $f_1, \ldots, f_N$ in some projective space. The coefficients of these polynomials lie in the field $\bigcup F_n$, but as there are only finitely many coefficients, they lie in a subfield $F_m$. We therefore see that the curve $X$ is the base change of a curve $X'$ defined over $F_m$, and if we set $K' = F_m(X')$, we have $K = K' \cdot_{F_m} (\bigcup F_n)$ as in the definition.  
Similar arguments yield the following examples as well:

\begin{ex}
Let $\clK(F)$ be the collection of field extensions of the form $F(X)$, where $X/F$ is a smooth projective variety admitting a rational point over $F$. Then $\clK(F)$ is regularly compact.   
\end{ex}

\begin{ex} \label{finite example}
Let $\clK(F)$ be the collection of finite field extensions of F. Then $\clK(F)$ is regularly compact.   
\end{ex}

\begin{lem} \label{brauer compact}
Let $\clK$ be any regularly compact class of fields and $\alpha \in \Br(k)$. Then the class $\clK_\alpha$ defined by $\clK_\alpha(F) = \{K \in \clK(F) \mid \alpha_K = 0\}$ is also regularly compact.
\end{lem}
\begin{proof}
Indeed, if $K \in \clK(\bigcup F_n)$ and $\alpha_K = 0$, we find that by definition we may find $m$ so that $K = K' \cdot_{F_m} (\bigcup F_n)$ for some $m$, and $K' \in \clK(F_m)$. But note that we can also write $K = \bigcup_{n \geq m} K' \cdot_{F_m} F_n$. Let $X_\alpha$ be a Severi-Brauer variety of some central simple algebra representing the class $\alpha$. As $\alpha_K = 0$, we see that $X_\alpha(K) \neq \emptyset$. As the finitely many coordinates of such a rational point must lie on one of the fields $K' \cdot_{F_m} F_n$, we find that $K' \cdot_{F_m} F_n \in \clK_\alpha(F_n)$ and $K = K' \cdot_{F_m} (\bigcup F_n) = (K' \cdot_{F_m} F_n) \cdot_{F_n} (\bigcup F_n)$ as desired.
\end{proof}

\begin{lem} \label{technical prop}
Suppose $\ms A, \ms B \subset \Br(k)$ are subsets of the Brauer group, $\clK$ and is a regularly compact class of field extensions of $k$. Suppose that for every $F/k$ and $\alpha \in \ms A, \beta \in \ms B, K \in \clK(F)$ with $\alpha_K = 0$, we may construct an inert extension $F'$ with $\beta_{K \cdot_F F'} = 0$. Then we may construct an inert extension $\til F/F$ with $\ms A_{\til F} \prec_{\clK(\til F)} \ms B_{\til F}$.
\end{lem}

To prove this, we will make use of the following Lemma, 
which follows a well known pattern. We include a proof for convenience.
\begin{lem} \label{limit injective}
Let $k$ be a field, $E/k$ a field extension, $I$ a filtered index set (i.e. a directed partially ordered set) and suppose we have a collection of field extensions $E_i$ for $i \in I$ such that $E_i \subset E_j$ for $i \leq j$ with $\bigcup_{i \in I} E_i = E$. If $E_i/k$ is inert for each $i \in I$ then $E/k$ is also inert.
\end{lem}
\begin{proof}
Suppose the map $\Br(k) \to \Br(E)$ was not injective and let $\beta$ be in the kernel.
Choose a central simple $k$-algebra $B$ representing the class of $\beta$, and an isomorphism $B \otimes_k E \simeq End(V)$. This isomorphism only will require finitely many field elements of $E$ to express, and hence can only involve elements in finitely many of the subfields $E_i$.
Letting $j$ be the maximum this finite set of such $i$, we find that $\beta_{E_{j}} = 0$ contradicting the fact that $E_{j} / k$ is inert.
\end{proof}

\begin{proof}[Proof of \cref{technical prop}]
As in the statement, let us suppose we have constructed inert extensions $F' = F(K, \alpha, \beta)/F$ with $\beta_{K \cdot_F F'} = 0$ for every $F/k$ and $\alpha \in \ms A, \beta \in \ms B, K \in \clK(F)$ with $\alpha_K = 0$. We first will show that we can construct an inert field extension $F(\ms A, \ms B)/F$ such that for every $K \in \clK(F)$ with $\alpha_K = 0$ for some $\alpha \in \ms A$, we have $\beta_{K \cdot_F F(\ms A, \ms B)} = 0$. We construct this field as follows. Let $\ms T \subset \clK(F) \times \ms A \times \ms B$ be those triples $(K, \alpha, \beta)$ such that $\alpha_K = 0$. Choose a well ordered set $\Omega$ with an initial element $\emptyset$, and a bijection 
\begin{align*}
\phi: \Omega_+ &\to \ms T,\\
\omega &\mapsto (K_\omega, \alpha_\omega, \beta_\omega) \in \clK(F) \times \ms A \times \ms B.
\end{align*}
 where  $\Omega_+ = \Omega \setminus \{\emptyset\}$. We will define an increasing collection of inert field extensions $F_\omega$ (transfinite) inductively as follows. We begin by $F_\emptyset = F$. 
 In general, we define $F'_\omega = \bigcup_{\lambda < \omega} F_\lambda$, and as these are inert and in particular regular over $F$, we find $K \cdot_F F'_\omega \in \clK(F'_\omega)$ and $(\alpha_\omega)_{K \cdot_F F'_\omega} = ((\alpha_\omega)_K)_{K \cdot_F F'_\omega} = 0$. Therefore, by hypothesis, there exists an inert extension $F_\omega = F'_\omega(K \cdot_F F'_\omega, \alpha_\omega, \beta_\omega)$ such that $(\beta_\omega)_{F_\omega} = 0$. We let $F(\ms A, \ms B) = \bigcup_{\omega \in \Omega} F_\omega$, which is inert by \cref{limit injective}.
 
 Let us check that whenever $K \in \clK(F)$ with 
 $\alpha \in \ms A$ and $\alpha_K = 0$, we have $\beta_{K \cdot_F F(\ms A, \ms B)}$ as claimed. 
 For such a choice $(K, \alpha, \beta) \in \ms T$, say 
 $(K, \alpha, \beta) = (K_\omega, \alpha_\omega, \beta_\omega)$ for some $\omega \in \Omega$. By construction, we have $\beta_{K \cdot_F F'_\omega} = 0$ and $F'_\omega \subset F(\ms A, \ms B)$. Therefore we have $K \cdot_F F'_\omega \subset K \cdot_F F(\ms A, \ms B)$ showing $\beta_{K \cdot_F F(\ms A, \ms B)} = 0$.
 
 We complete the proof with a second inductive procedure, this time constructing a sequence of fields $F_n$ for $n \in \mathbb Z_{\geq 0}$, starting with $F_0 = F$ and with $F_{n+1}/F_n$ inert by setting $F_{n + 1} = F_n(\ms A_{F_n}, \ms B_{F_n})$. Let $\til F = \bigcup F_n$, which is inert by \cref{limit injective}. Let us now check that $\ms A_{\til F} \prec_{\clK(\til F)} \ms B_{\til F}$. For this, suppose $K \in \clK(\til F)$ with $\alpha_K = 0$ for some $\alpha$ and let $\beta \in \ms B$. We claim that $\beta_K = 0$ as well.
 
 By the regular compactness of $\clK_\alpha$ from \cref{brauer compact}, we may find $m$ and $K' \in \clK(F_m)$ with $K = K' \cdot_{F_m} \til F$ and $\alpha_{K'} = 0$. 
 But therefore $\beta_{K' \cdot_{F_m} F_m(\ms A, \ms B)} = \beta_{K' \cdot_{F_m} F_{m + 1}} = 0$. 
 But as $K' \cdot_{F_m} F_{m + 1} \subset K' \cdot_{F_m} \til F = K$, we have 
 $\beta_K = 0$ as desired.
\end{proof}

\section{Ordering of Brauer classes by splitting}

In this section, we now complete our construction of collections of algebras with common finite splitting fields and common splitting fields arising as function fields of curves of genus at most $1$ (see \cref{forcing 1}).

We consider the following classes of field extensions of a field $k$:
\begin{itemize}
\item  
$\underline{\ms G}^{0}(F) = \{F(X) \mid X/F \text{ is a smooth projective curve of genus $0$}\}$,
\item  
$\underline{\ms G}^{1}(F) = \{F(X) \mid X/F \text{ is a smooth projective curve of genus $1$}\}$,
\item  
$\underline{\ms G}^{\leq 1}(F) = \{F(X) \mid X/F \text{ is a smooth projective curve of genus at most $1$}\}$,
\item 
$\underline{\ms Fin}(F) = \{K/F \text{ finite}\}$,
\item
$\underline{\ms G^1\!\ms F in}(F) = 
\underline{\ms G}^{1}(F) \cup \underline{\ms Fin}(F)$,
\item
$\underline{\ms G^{\leq 1}\!\ms F in}(F) = 
\underline{\ms G}^{\leq 1}(F) \cup \underline{\ms Fin}(F)$,
\end{itemize}
and following the ideas of \cref{genus 1 example} and  \cref{finite example} it is not hard to see that all of these classes are regularly compact.

Let us start by briefly describing the situation concerning splitting algebras by function fields of genus $0$ curves.

\begin{prop} \label{genus 0}
Let $\alpha, \beta \in \Br(k)$ with $\alpha \neq 0$, and suppose $C$ is a genus $0$ curve with $k(C)$ splitting $\alpha$. Then $\alpha$ must be represented by a quaternion algebra and if $k(C)$ also splits $\beta$ then either $\beta = 0$ or $\beta = \alpha$. In particular, we find that $\ms A \prec_{\underline{\ms G}^{0}(k)} \ms B$ always holds (trivially) when $\ms A$ consists entirely of classes of period greater than $2$.
\end{prop}
\begin{proof}
As curves of genus $0$ are plane conics, they can be identified with Severi-Brauer varieties of quaternion algebras, say $C = X_Q$ for some quaternion algebra $Q$. As $\Br(k(X_Q)/k)$ consists of $0$ and the class of $Q$, it follows that $[Q] = \alpha$ and $\beta \in \Br(k(X_Q)/k)$ is $0$ or $\alpha$ as claimed.
\end{proof}

\begin{thm} \label{forcing 1}
Let $\ms A_1, \ms A_2 \subset \Br(k)$ be collections of Brauer classes such that $\per(\alpha_2) | \per(\alpha_1)$ whenever $\alpha_1 \in \ms A_1, \alpha_2 \in \ms A_2$. Then we may find an inert field extension $F/k$ such that $(\ms A_1)_F \precspl{\underline{\ms G^1\!\ms F in}(F)} (\ms A_2)_F$.

Consequently, given any collection of Brauer classes $\ms A$, with all elements having the same period $n > 1$, we may find an inert extension $F/k$ such that all the classes in $\ms A$ share the same set of splitting fields in $\underline{\ms G^1\!\ms F in}(F)$. Furthermore, if $n > 2$, we find that all the classes in $\ms A$ share the same set of splitting fields in $\underline{\ms G^{\leq 1}\!\ms F in}(F)$ as well.
\end{thm}

Note that the latter statement follows from taking $\ms A_1 = \ms A = \ms A_2$, combined with \cref{genus 0}.

\begin{question} \label{question}
Can \cref{forcing 1} be extended to curves of higher genus as well?
\end{question}

To prove \cref{forcing 1}, we start with a few preliminary results.

\begin{lem} \label{genus 1 geom}
Suppose $\alpha_1, \alpha_2 \in \Br(k)$ with $\per \alpha_2 | \per \alpha_1$, and $C/k$ is a genus $1$ curve with $(\alpha_1)_{k(C)} = 0$. Then there exists an inert extension $E/k$ with $(\alpha_2)_{E(C)} = 0$.
\end{lem}
\begin{proof}
Write $b$ for the pointed Brauer class $(\beta_C|_{J_C}, J_C, 0)$. By \cref{canonical obstruction}, since $\alpha_1{k(C)} = 0$, it follows $\alpha_1 \in b$ is a specialization, and hence $\per \alpha_1 | \per b$ and consequently $\per \alpha_2 | \per b$ as well.
By \cref{eventually special prop}, we may find an inert field extension $E/k$ such that $(\alpha_2)_E \in b_E$. But $(\alpha_2)_E \in b_E$ means $\alpha_2$ is a specialization of $(\beta_C)_E = \beta_{C_E}$ and hence $\alpha \in \Br(E(C)/E)$ again by Proposition~\ref{canonical obstruction}. 
\end{proof}
\begin{lem} \label{base 0 case}
Let $k$ be a field, $\alpha_1, \alpha_2 \in \Br(k)$ with $\per(\alpha_2) | \per(\alpha_1)$. Let $K/k$ be a finite field extension with $(\alpha_1)_K = 0$. Then we may find an inert field extension $E/k$ such that $(\alpha_2)_{K\cdot_k E} = (\alpha_2)_{K \otimes_k E} = 0$.
\end{lem}
\begin{proof}
Let $K$ be a finite splitting field of $\alpha_1$. By \cite[Lemma~7.4, Corollary~7.6]{KMRRS}, we may find a regular field extension $k_K(\alpha_2)$ such that the compositum $Kk_K(\alpha_2)$ (equal in this case to the tensor product) splits $\alpha_2$ and $\Br(k_K(\alpha_2)/k) = 0$ as claimed.
\end{proof}

\begin{proof}[Proof of \cref{forcing 1}]
By \cref{technical prop} we need only show that if we are given $\alpha_i \in \ms A_i$, $i = 1, 2$ and $K \in \underline{\ms G^1\!\ms F in}(F)$ for some field extension $F/k$, with $(\alpha_1)_K = 0$, we can find $E/F$ inert with $(\alpha_2)_{K \cdot_F E} = 0$ as well. But in the case $K \in \underline{\ms G}^{1}(F)$, this follows from \cref{genus 1 geom}, and in the case $K \in \underline{\ms Fin}(F)$ this follows from \cref{base 0 case}.
\end{proof}

\section{Finiteness of higher upper genera}


Let $\alpha_1, \alpha_2 \in \Br(k)$, and suppose $n = \ind(\alpha_1)$. Then
$\alpha_1 \precspl{k(n-1)} \alpha_2$ if and only if $\alpha_2$ is in the cyclic
subgroup generated by $\alpha_1$. This is because the function field of the
Severi-Brauer variety of $\alpha_1$ has transcendence degree $n-1$ and hence
must also split $\alpha_2$. But the only algebras split by this function field
lie in the cyclic subgroup generated by $\alpha_1$ by \cite{Am}.

On the other hand, for a class $\alpha \in \Br(k)$, and a given integer $i$, we
may still ask for which $\beta$ we have $\alpha \precspl{k(i)} \beta$. This
motivates the following definition:

\begin{defn}
Let $\alpha \in \Br(k)$. We define the $i$'th upper genus of $\alpha$, denoted
$[\alpha, \infty)_i$ to be the set of $\beta \in \Br(k)$ such that $\beta
\precspl{k(i)} \alpha$.
\end{defn}

\begin{defn}
Let $\alpha \in \Br(k)$. We say that $\alpha$ is $i$-minimal if $[\alpha,
\infty)_i$ consists precisely of the cyclic subgroup of $\Br(k)$ generated by
$\alpha$.
\end{defn}

\begin{prop} \label{3 splitting} Given a class $\alpha \in \Br(k)$, there is a
smooth projective variety $Y$ of dimension at most $3$ such that for every field
extension $L/k$, $\Br(L(Y)/L)$ is exactly the cyclic subgroup generated by
$\alpha_L$.
\end{prop}
\begin{proof}
  If $k$ is finite there is nothing to prove, as $\Br(k)=0$. Thus, we assume that $k$ is infinite. 
  Consider the Severi-Brauer variety $X$ for $\alpha$. If $\dim X \leq 3$, we may set $Y = X$ and be done by \cite{Am}. Otherwise, let $Y\subset X$ be a smooth complete intersection of $\dim X-3$ anti-canonical divisors in $X$, which is possible by Bertini's theorem (since $k$ is infinite). 
  
  We will show that the map of group schemes
  \begin{equation}\label{resscheme}
  \Pic_X\to\Pic_Y
  \end{equation}
  is an isomorphism. It follows that 
  map \begin{equation*}\label{resmap}
      \mathbb Z=\Pic_X(k)\to\Pic_Y(k)
  \end{equation*} 
  is an isomorphism. Since the relative Brauer group is generated by obstructions associated to sections of the Picard scheme and $\Br(k(X)/k)=\langle\alpha\rangle$, we conclude that $\Br(k(Y)/k)=\langle\alpha\rangle$. Since our statements are independent of $L$, this gives the desired conclusion.
  
  For technical reasons, we briefly recall that the deformation theory of invertible sheaves on a scheme $Z$ has tangent space $H^1(Z,\ms O_Z)$ and obstruction space $H^2(Z,\ms O_Z)$. In particular, if $Z\to S$ is proper, of finite presentation and cohomologically flat in dimension $0$, and if $H^i(Z_s,\ms O_{Z_s})=0$ for $i=1,2$ and all points $s\in S$, then $\Pic_{Z/S}$ is \'etale. Moreover, if, in addition, $Z$ is smooth over $S$ and $S$ is integral, then $\Pic_{Z/S}$ is also universally closed over $S$. 
  
  This has two consequences. First, to show that \eqref{resscheme}
  is an isomorphism of group schemes, it suffices to show that the map
  $$\Pic(X\otimes \overline k)\to\Pic(Y\otimes\overline k)$$
  is an isomorphism of groups. This means that we can reduce the desired claim to the case that $k$ is algebraically closed, and hence that $X=\mathbb P^n$.
  
  Second, the deformation theory calculations allow us to reduce the question to characteristic $0$. Assume that $W$ is a complete local ring with fraction field of characteristic $0$, and let $\mathcal Y\subset\mathbb P^n_W$ be a flat complete intersection of dimension at least $3$. Consider the restriction map 
  $$\Pic_{\mathbb P^n/W}\to\Pic_{\mathcal Y/W}.$$
  By the above considerations, this is an isomorphism if and only if it is an isomorphism over the algebraic closure of the fraction field of $W$. But this follows from the Grothendieck-Lefschetz Theorem \cite[Corollary~3.3]{Hart:Ample}.

  Finally, let $W$ be the Witt vectors of $k$ and assume that $X=\mathbb P^n$. Since $Y$ is a complete intersection we can lift it to a smooth relative complete intersection $\mathcal Y\subset\mathbb P^n_W$. By the previous paragraph, we conclude that \eqref{resscheme} is an isomorphism, as desired.
\end{proof}

\begin{prop} \label{2 exact splitting}
Let $\alpha \in \Br(k)$ be a Brauer class, and suppose that $k$ is an uncountable field of characteristic $0$. Then we may find a smooth projective variety $Z$ of dimension at most $2$ such that for every field extension $L/k$, $\Br(L(Z)/L)$ is exactly the cyclic subgroup generated by $\alpha_L$.
\end{prop}
\begin{proof}
Via Proposition~\ref{3 splitting}, without loss of generality, we may assume that we have $Y$ of dimension $3$ such that $\Pic(Y_{\overline{k}}) = \mathbb Z$, whose generator corresponds to an $\alpha$-twisted sheaf. 
If $Y_{\overline{k}}$ is a Severi-Brauer variety, then by Noether-Lefschetz,
if $Z \subset Y$ is a very general hyperplane section with respect to a projective embedding of $Y$, pullback gives an isomorphism $\Pic Y \to \Pic Z$.

In the case $Y$ is not a Severi-Brauer, choose an ample line bundle $\ms L$ on $Y$. By \cite[Theorem~2]{LM:NoethLef}, if we choose $d$ sufficiently large, then $Z$ cut out by a very general section of $\ms L^{\otimes d}$, we have $\Pic Y \to \Pic Z$ is again an isomorphism.

It follows again that $\Pic Z$ is generated by an $\alpha$-twisted sheaf, and hence the relative Brauer group $\Br(k(Y)/k)$ is exactly the cyclic subgroup generated by $\alpha$.
\end{proof}

\begin{prop} \label{2 splitting}
Suppose that $\alpha$ is a Brauer class over a field $k$ of characteristic $0$.
Then we may find a smooth projective surface $Y/k$ such that $\alpha \in \Br(Y/k)$ and $\Br(Y_L/L)$ is finite for every field extension $L/k$.
\end{prop}
\begin{proof}
Without loss of generality, we may assume $k$ is infinite.
Consider the Severi-Brauer variety $X$ for $\alpha$. If $\dim X \leq 2$, we may set $Y = X$ and be done by \cite{Am}. Otherwise, fix a projective embedding $X \to \mathbb P^N$ for some $N$. Let $Y \subset X$ be a smooth linear section of $X$ of dimension $2$ (we can find such a subvariety since $k$ is infinite).
By the Lefschetz hyperplane Theorem \cite[page~156]{GrifHar}, the map the first Betti numbers of $X_{\mathbb C}$ and $Y_{\mathbb C}$ coincide and hence vanish (where $\mbb C$ denotes an algebraic closure of $k$).
It follows that the Albanese variety for $Y$, and hence the Picard variety for $Y$ must be trivial, which implies that
the Picard group and Neron-Severi groups for $Y$ agree. In particular, $Pic(Y)$ is a finitely generated group \cite[Theorem~5.1, Exp~XIII, p.~650]{SGA6}. If $\beta$ is the canonical obstruction class for $Y$, then for any field extension $L/k$, the class $\beta_L$ is determined by its values on the generators of $\Pic_{Y}(L)$, which are finite. The relatives Brauer group $\Br(Y_L/L)$ is, by Proposition~\ref{canonical obstruction}, generated by the values of $\beta$ at these generators, and therefore is finite.
\end{proof}

\begin{thm} \label{3 genus}
Let $\alpha \in \Br(k)$ be a nontrivial Brauer class and $k$ a field of characteristic $0$. Then $[\alpha, \infty)_3$ is always finite and equal to the cyclic subgroup generated by $\alpha$. If $k$ is also uncountable, then $[\alpha, \infty)_2 = \left< \alpha \right>$ as well.
\end{thm}
\begin{proof}
By Proposition~\ref{3 splitting}, we may find $Y$ of dimension at most $3$ such that $\Br(k(Y)/k) = \left<\alpha\right>$. But therefore $[\alpha, \infty)_3 \subset \Br(k(Y)/k)$.
But since $[\alpha, \infty)_3$ is a subgroup of $\Br(k)$ containing $\alpha$, the result follows.
The case of an uncountable field follows similarly from Proposition~\ref{2 exact splitting}.
\end{proof}

\begin{thm} \label{2 genus}
Let $\alpha \in \Br(k)$ be a nontrivial Brauer class over a field $k$. Then $[\alpha, \infty)_2$ is always finite.
\end{thm}
\begin{proof}
By Proposition~\ref{2 splitting}, we may find $Y$ of dimension at most $2$ such that $\Br(k(Y)/k)$ is finite and contains $\alpha$.
But since $[\alpha, \infty)_2$ is contained in $\Br(k(Y)/k)$, the result follows.
\end{proof}

\bibliographystyle{alpha}
\newcommand{\etalchar}[1]{$^{#1}$}
\def\cprime{$'$} \def\cprime{$'$} \def\cprime{$'$} \def\cprime{$'$}
  \def\cftil#1{\ifmmode\setbox7\hbox{$\accent"5E#1$}\else
  \setbox7\hbox{\accent"5E#1}\penalty 10000\relax\fi\raise 1\ht7
  \hbox{\lower1.15ex\hbox to 1\wd7{\hss\accent"7E\hss}}\penalty 10000
  \hskip-1\wd7\penalty 10000\box7}

\end{document}